\documentclass[a4paper,UKenglish]{lipicsplain}
 
\usepackage{microtype}

\usepackage{amsmath,epsf,amssymb,cite,amsthm, mathrsfs,epsfig,  bbm, amsthm,   setspace, bbold}

\usepackage{etex}

\usepackage{graphicx}
\usepackage{caption}
\usepackage{hyperref}
\hypersetup
{
    colorlinks,	%
    citecolor=purple,%
    filecolor=black,%
    linkcolor=purple,%
    urlcolor=purple
}

\usepackage{Liz}

\usepackage{tikz-cd}

\newcommand{\inter}{d_I}

\graphicspath{{Figures/}}

\title{Strong Equivalence of the Interleaving and Functional Distortion Metrics for Reeb Graphs}
\titlerunning{Strong Equivalence of Reeb Graph Metrics} 

\author[1]{Ulrich Bauer}
\author[2]{Elizabeth Munch}
\author[3]{Yusu Wang}
\affil[1]{Department of Mathematics,
  Technische Universit\"at M\"unchen\\
  \texttt{mail@ulrich-bauer.org}}
\affil[2]{Department of Mathematics \& Statistics,
  University at Albany -- SUNY\\
  \texttt{emunch@albany.edu}}
\affil[3]{Department of Computer Science and Engineering, The Ohio State University\\
  \texttt{yusu@cse.ohio-state.edu}}
\authorrunning{U.~Bauer, E.~Munch and Y.~Wang} 

\Copyright{Ulrich Bauer, Elizabeth Munch, and Yusu Wang}


\serieslogo{}
\volumeinfo
  {Billy Editor and Bill Editors}
  {2}
  {Conference title on which this volume is based on}
  {1}
  {1}
  {1}
\EventShortName{}
\DOI{10.4230/LIPIcs.xxx.yyy.p}

\date{20th December 2014}

\begin{document}

\maketitle

\begin{abstract}
The Reeb graph is a construction that studies a topological space through the lens of a real valued function.  It has widely been used in applications, however its use on real data means that it is desirable and increasingly necessary to have methods for comparison of Reeb graphs. Recently, several methods to define metrics on the space of Reeb graphs have been presented.  In this paper, we focus on two: the functional distortion distance and the interleaving distance. The former is based on the Gromov--Hausdorff distance, while the latter utilizes the equivalence between Reeb graphs and a particular class of cosheaves. However, both are defined by constructing a near-isomorphism between the two graphs of study.  In this paper, we  show that the two metrics are strongly equivalent on the space of Reeb graphs.  In particular, this gives an immediate proof of bottleneck stability for persistence diagrams in terms of the Reeb graph interleaving distance.

 \end{abstract}

\section{Introduction}

The Reeb graph is a construction that can be used to study a topological space with a real valued function by tracking the relationships between connected components of level sets. 
It was originally developed in the context of Morse theory \cite{Reeb1946}, and was later introduced for shape analysis by Shinagawa et al.~\cite{Shinagawa1991}. 
Since then, it has attracted much attention due to its wide use for various data analysis applications, such as shape comparison \cite{Hilaga2001, Escolano2013}, denoising \cite{Wood2004}, and shape understanding \cite{Dey2013,HA03}; see \cite{Biasotti2008} for a survey. 
Recently, the applications of Reeb graphs have been further broadened to summarizing high-dimensional and/or complex data, in particular, reconstructing non-linear 1-dimensional structure in data \cite{NBPF11,Ge2011,Chazal2014} and summarizing collections of trajectory data \cite{BBKSS13}. 
Its practical applications have also been facilitated by the availability of efficient algorithms for computing the Reeb graph from a piecewise-linear function defined on a simplicial complex  \cite{Parsa2012,Harvey2010,Doraiswamy2012}. 

In addition to the standard construction, a generalization of the Reeb graph construction, known as Mapper, \cite{Singh2007}, has proven extremely useful in the field of topological data analysis \cite{Yao2009, Nicolau2011}.
A variant of Mapper for real-valued functions, called the $\alpha$-Reeb graph, was used in \cite{Chazal2014} to study data sets with 1-dimensional structure.

Given the popularity of the Reeb graph and related constructions for practical data analysis applications, it is desirable and increasingly necessary to understand how robust (stable) these structures are in the presence of noise. 
Consequently, several metrics for comparing Reeb graphs have been proposed recently. 
These include the interleaving distance \cite{deSilva2014}, the functional distortion distance \cite{Bauer2014}, and the combinatorial edit distance \cite{DiFabio2014b}. 
We note that the latter is limited to Reeb graphs resulting from Morse functions defined on surfaces (2-manifolds). 
In addition, Morozov et.~al proposed an interleaving distance for a simpler variant of the Reeb graph, the so-called merge tree \cite{Morozov2013}. 

In this paper, we study the relation between the only two distance measures for general Reeb graphs proposed in the literature:  the functional distortion distance of \cite{Bauer2014} and the interleaving distance of \cite{deSilva2014}. 
The former is based on concepts from metric geometry, and is defined by treating both graphs as metric spaces and inspecting continuous maps between them. 
The latter, on the other hand, is defined using ideas of category theory, utilizing the equivalence between Reeb graphs and a particular class of cosheaves. 
However, in essence, both construct a near-isomorphism between the two input graphs of study. 
In Sections \ref{sec:relation} and \ref{sec:bottleneck}, we explore this connection between the two distances, and show that indeed, the functional distortion distance and the interleaving distances are strongly equivalent on the space of Reeb graphs, meaning that they are within constant factor of each other. 
This immediately leads to the bottleneck stability result for the Reeb graph interleaving distance.

\section{Definitions}

Given a topological space $\X$ with a real valued function $f:\X \to \R$, we  define the Reeb graph of $(\X,f)$ as follows.
We say that two points in $\X$ are equivalent if they are in the same path-connected  component of a level set $f\inv(a)$ for $a \in \R$.
This is notated $x \sim_f y$ or $x \sim y$ if the function is obvious. 
Then the Reeb graph is the quotient space $\X/\sim_f$. 
Note that the Reeb graph inherits a real valued function from its parent space.
See Fig.~\ref{Fig:BasicExample} for an example.
\begin{figure}
\centering
 \includegraphics[scale = .4]{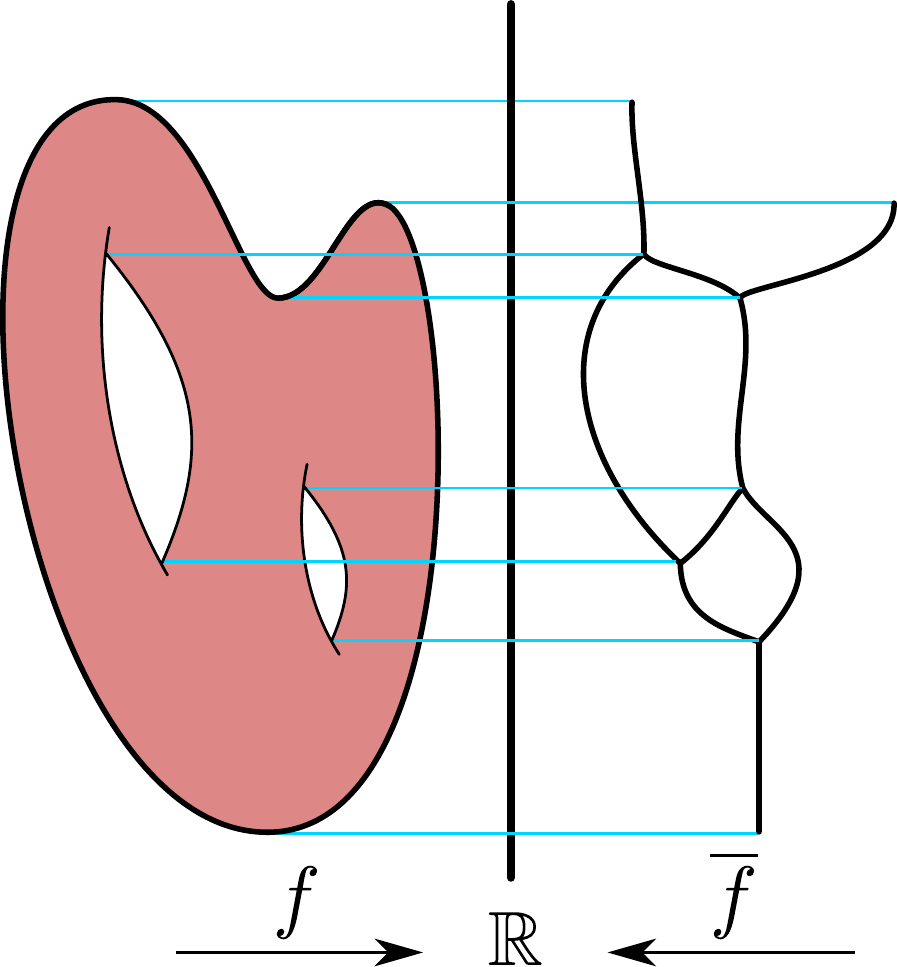}
 \caption{A simple example of the Reeb graph (right) of a space (left).  Here and in all other drawn examples in this paper, the real valued function is denoted by height.}
 \label{Fig:BasicExample}
\end{figure}

\subsection{Category of Reeb Graphs}
For nice enough functions $f:\X \to \R$, such as Morse functions on compact manifolds or PL functions on finite simplicial complexes, the Reeb graph is, in fact, a finite graph. We will tacitly make this assumption on functions throughout the paper.
Thus, we will define the category of Reeb graphs, following \cite{deSilva2014}, intuitively to be finite graphs with real valued functions.
Morphisms will be given by function preserving maps between the underlying spaces as given in the following definition.

\begin{definition}
\label{Def:ReebCat}
An object of the category $\Reeb$ is a finite graph, seen as a topological space, more precisely as a regular CW complex, together with a real valued function that is monotone on edges.  
This will equivalently be written as either $f:\X \to \R$ or $(\X,f)$.
A morphism between $(\X,f)$ and $(\Y,g)$ is a map $\phi:\X \to \Y$ preserving function values, i.e., the following diagram commutes:
\begin{equation*}
\begin{tikzcd}
 \X \ar{r}{\phi} \ar{dr}{f}& \Y \ar{d}{g}\\
 & \R
\end{tikzcd}
\end{equation*}
\end{definition}
 
Note that because we assume that the function is monotone when restricted to the edges, we will often just think of the function as being given by the values on the vertices.
As an aside, notice that the process of taking the Reeb graph of a space with a function is an isomorphism in $\Reeb$.

\subsection{Interleaving Distance}

Given a Reeb graph $(\X,f)$, let $\X_\e$ denote the space $\X \times [-\e,\e]$ and define the $\e$-smoothing of $(\X,f)$ as the Reeb graph of the function
\begin{equation*}
 \begin{array}{cccc}
 f_\e: & \X_\e &\to &\R\\
  & (x,t) & \longmapsto & f(x) + t.
 \end{array}
\end{equation*}
That is, the $\e$-smoothing is the quotient space $\X_\e/\sim_{f_\e}$.
Denote this space by $\TT_\e(\X,f)$ and notice that $\TT_\e(\TT_\e(\X,f)) \iso \TT_{2\e}(\X,f)$.
Sometimes when we are focused on the underlying topological space, we will notat this as $\TT_\e(\X)$.
See Fig.~\ref{Fig:SmoothedExample} for an example.

\begin{figure}
 \centering
 \includegraphics[scale = .8]{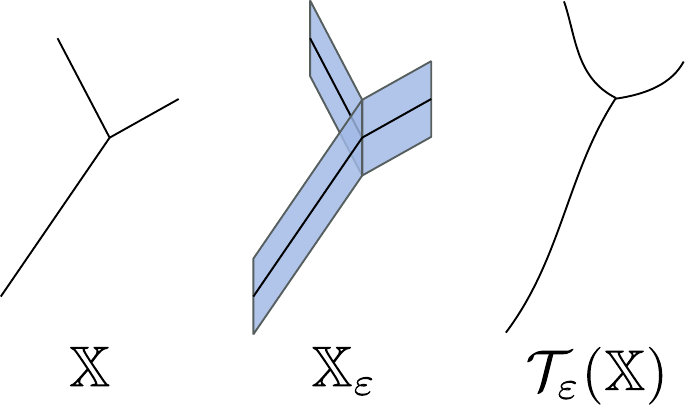}
 \caption{An example of a smoothed Reeb graph.  At left is the original graph $\X$ with function $f$ given by height.  The middle space is $\X_\e = \X \times [-\e,\e]$ with the function $f_\e(x,t) = f(x)+t$ still given by height.  At right is the Reeb graph of $(\X_\e,f_\e)$, which is the smoothed Reeb graph $\TT_\e(\X)$.}
 \label{Fig:SmoothedExample}
\end{figure}

An $\e$-interleaving of $(\X,f)$ and $(\Y,g)$ is a pair of function preserving maps $\phi:(\X,f) \to \TT_\e(\Y, g)$ and $\psi:(\Y,g) \to \TT_\e(\X, f)$ with the following requirements.
First, note that there is a natural map $\iota: (\X,f) \to \TT_\e(\X,f)$ which sends $x\in \X$ to the equivalence class of $(x,0)$ in $\X_\e/\sim_{f_\e}$.
Similarly, there is another natural map on $(\Y,g)$, which we abuse notation and also denote $\iota:(\Y,g) \to \TT_\e(\Y,g)$.
There exists a map $\phi_\e:\TT_\e(\X,f) \to \TT_{2\e}(\Y,g)$, uniquely defined on $\im \iota$ by the map $\phi$, such that the diagram
\begin{equation*}
\begin{tikzcd}
 (\X,f) \ar{r}{\iota}\ar{d}{\phi} &  \TT_\e(\X,f)\ar{d}{\phi_\e} \\
 \TT_\e(\Y,g)  \ar{r}{\iota_\e} & \TT_{2\e}(\Y,g)
\end{tikzcd}
\end{equation*}
commutes.
The map $\psi_{\e}:\TT_\e(\Y,g) \to \TT_{2\e}(\X,f)$ is defined similarly.
Then we have the definition of the $\e$-interleaving.

\begin{definition}[Geometric $\e$-Interleaving]
The maps $\phi:(\X,f) \to \TT_\e(\Y,g)$ and  $\psi:(\Y,g) \to \TT_\e(\X,f)$ are a \emph{(geometric) $\e$-interleaving} if the diagram
\begin{equation*}
\begin{tikzcd}[/tikz/commutative diagrams/row sep=large]
 (\X,f) \ar{r}{\iota} \ar{dr}[very near start]{\phi}  & 
 \TT_\e(\X,f) \ar{r}{\iota_\e} \ar{dr}[very near start]{\phi_\e} & 
 \TT_{2\e}(\X,f) \\
 (\Y,g) \ar{r}{\iota}  \ar{ur}[very near start]{\psi} & 
 \TT_\e(\Y,g) \ar{r}{\iota_\e} \ar{ur}[very near start]{\psi_\e} & 
 \TT_{2\e}(\Y,g)
\end{tikzcd}
\end{equation*}
commutes.
\end{definition}

We can use this definition of interleavings to define a distance on Reeb graphs.
\begin{definition}[Interleaving Distance, \cite{deSilva2014}]
The \emph{interleaving distance} between two Reeb graphs $(\X,f)$ and $(\Y,g)$ is defined to be
\begin{equation*}
\inter(\Ffunc,\Gfunc) =
\inf\left\{ \e \mid \text{there exists an $\e$-interleaving between $(\X,f),(\Y,g)$} \right\}.
 \end{equation*}
\end{definition}

The definition of the interleaving distance was motivated by the cosheaf structure of Reeb graphs.
It was shown in \cite{deSilva2014} that the category of Reeb graphs is equivalent to a particular class of cosheaves, which can be thought of as functors $F:\Int \to \Set$ giving a set for each open interval.
The idea is that given a Reeb graph $f = (\X,f)$, we can construct the associated cosheaf $\Ffunc = \pi_0 \circ f\inv$.
This equivalence allows us to work with either the topological construction or the category theoretic, whichever is easier or more appropriate.
Specifics for this can be found in Appendix \ref{Sect:Cosheaves}, and an excellent introduction can be found in \cite{Curry2013}.

Let the $\e$-thickening of an interval $I = (a,b)$ be denoted by $I^\e = (a-\e,b+\e)$.
Then we can also consider the $\e$-interleavings for the cosheaves.
\begin{definition}[Cosheaf $\e$-Interleaving]
Given Reeb graphs $(\X,f)$ and $(\Y,g)$, 
an {\bf $\e$-interleaving} between the corresponding cosheaves $\Ffunc = \pi_0\circ f\inv$ and $\Gfunc = \pi_0 \circ g\inv$ is given by two families of maps
\begin{align*}
\phi_I:& \pi_0\circ f\inv(I) \to \pi_0\circ g\inv(I^\e),\\
\psi_I:&  \pi_0\circ g\inv(I) \to  \pi_0\circ f\inv(I^\e)
\end{align*}
for each open interval $I$.
These must be natural with respect to inclusions $I \subseteq J$, i.e., the following diagrams commute, where the  vertical maps are induced by the inclusions $\iota: I \hookrightarrow J$ and $\iota_\e: I^\e \hookrightarrow J^\e$:
\begin{equation*}
 \begin{tikzcd}
 \pi_0(f\inv(I)) \ar{r}{\phi_I} \ar{d}{\iota^*}&  \pi_0(g\inv(I^\e))  \ar{d}{\iota^*} &&
 \pi_0(g\inv(I)) \ar{r}{\psi_I} \ar{d}{\iota^*}&  \pi_0(f\inv(I^\e))  \ar{d}{\iota^*}\\
 \pi_0(f\inv(J)) \ar{r}{\phi_J}&  \pi_0(g\inv(J^\e)) &&
 \pi_0(g\inv(J)) \ar{r}{\psi_J}&  \pi_0(f\inv(J^\e)) \\
 \end{tikzcd}
\end{equation*}
Moreover, we require that
\begin{equation*}
\psi_{I^\e} \circ \phi_I : \pi_0\circ f\inv(I) \to \pi_0\circ f \inv (I^\e)
\end{equation*}
is the map induced by the inclusions $f\inv(I) \subset f\inv(I^\e)$ and 
\begin{equation*}
\phi_{I^\e} \circ \psi_I : \pi_0\circ g\inv(I) \to \pi_0\circ g \inv (I^\e)
\end{equation*}
is the map induced by the inclusions $g\inv(I) \subset g\inv(I^\e)$
for all $I$. 
\end{definition}
First, notice that for $\e=0$ this is is precisely an isomorphism between $\Ffunc, \Gfunc$.
Note that an interleaving could be equivalently defined as a pair of natural transformations between the appropriate functors.
Then, the interleaving distance can be equivalently defined as 
\begin{equation*}
\inter((\X,f),(\Y,g)) =
\inf\left\{ \e \mid \text{there exists an $\e$-interleaving between $\Ffunc = \pi_0\circ f\inv,\Gfunc = \pi_0\circ g\inv$} \right\}
 \end{equation*}
 for two cosheaves $\Ffunc, \Gfunc: \Int \to \Set$.
We will therefore abuse notation and also denote this as $\inter((\X,f),(\Y,g))$.

\subsection{Functional Distortion Distance}

For a given path $\pi$ from $u$ to $v$ in $(\X,f) \in \Reeb$, we define the height of the path to be 
\begin{equation*}
 \textrm{height}(\pi) = \max_{x \in \pi} f(x) - \min_{x \in \pi} f(x).
\end{equation*}
Then we define the distance
\begin{equation*}
 d_f(u,v) = \min_{\pi:u \to v} \textrm{height}(\pi)
\end{equation*}
where $\pi$ ranges over all paths from $u$ to $v$ in $\X$.
Note that this can be equivalently be defined by the minimum length of an interval $I$ such that $u$ and $v$ are in the same connected component of $f\inv(I)$.  

The functional distortion distance between $(\X,f)$ and $(\Y,g)$ is now defined as follows:
\begin{definition}
 [Functional Distortion Distance, \cite{Bauer2014}]
 \label{Def:FuncDistortion}
 
 Given $(\X,f),(\Y,g) \in \Reeb$ and maps $\phi:\X \to \Y$ and $\psi:\Y\to \X$, let 
 \begin{equation*}
  C(\phi,\psi) = \{ (x,y) \in \X \times \Y \mid \phi(x) = y \textrm { or } x = \psi(y)\}
 \end{equation*}
 and 
 \begin{equation*}
  D(\phi,\psi)  = \sup_{\substack{(x,y),(x',y')\\ \in C(\phi,\psi)}} \frac{1}{2} \left| d_f(x,x') - d_g(y,y')\right|.
 \end{equation*}
 Then the \emph{functional distortion distance} is defined to be 
 \begin{equation*}
  d_{FD}(f,g) = \inf_{\phi,\psi} \max \{ D(\phi,\psi),\|f-g \circ \phi\|_\infty, \|g-f \circ \psi\|_\infty \}.
 \end{equation*}

\end{definition}
Note that since the maps $\phi,\psi$ are not required to preserve the function values, they are not Reeb graph morphisms in the sense of Definition~\ref{Def:ReebCat}.


\section{Multivalued Maps and Continuous Selections}

A multivalued map (or multimap) $F: X \to Y$ is a correspondence which sends a point $x \in X$ to a nonempty set $F(x) \subset Y$.
A selection of a multimap is a singlevalued function $f:X \to Y$ such that $f(x) \in F(x)$ for every $x \in X$.
See  \cite{Repovs1998} for an introduction to multimaps.

Note that using the axiom of choice, a selection always exists; the trick is to find a continuous selection.
The Michael selection theorem gives a criterion for a multimap to have a continuous selection.
However, in order to state it, we will need several definitions.

\begin{definition}
 A family $\SS$ of subsets of a topological space $Y$ is equi-locally $n$-connected if for every $S \in \SS$, every $y \in S$, and every neighborhood $W$ of $y$, there is a neighborhood $V$ of $y$ such that $V \subset W$ and for every $S' \in \SS$ such that $V \cap S' \neq \emptyset$, every continuous mapping of the $m$-sphere $\Sbb^m$ into $S' \cap V$ is null-homotopic in $S' \cap W$ for $m \leq n$. This is notated $\SS \in \ELC^n$.
\end{definition}
In particular, we will be needing the case where $\SS \in \ELC^0$, so the final requirement amounts to checking that $S'\cap V$ is path connected.

\begin{definition}
A multivalued map $F: X \to Y$ is lower semicontinuous (LSC) if for every open $U \subset Y$, the set $F\inv(U) = \{x \in X \mid F(x) \cap U \neq \emptyset\}$ is open in $X$.
\end{definition}

Finally we can state the Michael selection theorem.  
Note that we are working with a space of covering dimension 1, so we paraphrase the more general theorem here to relate it to our context.
\begin{theorem}[Michael 1956\cite{Michael1956}]
\label{Thm:Selection}
 A multivalued mapping $F : X \to Y$ admits a continuous single-valued selection provided that the following conditions are satisfied:
 \begin{enumerate}
  \item $X$ is a paracompact space with $\dim(X) \leq 1$;
 \item $Y$ is a completely metrizable space;
 \item $F$ is an LSC mapping;
 \item 
 For every $x \in X$, $F(x)$ is a 0-connected (path connected) subset of $Y$; and
 \item 
 The family of values $\{F(x)\}_{ x \in X}$ is $\ELC^0$.
 \end{enumerate}

\end{theorem}

\section{$\e$-Interleaving and Functional Distortion}
\label{sec:relation}

In order to prove the main result, Theorem~\ref{Thm:Main}, we will prove each inequality separately as Lemmas \ref{Lem:Inter leq FD} and \ref{Lem:FD leq Inter} .
\subsection{The Easy Direction}

\begin{lemma}
\label{Lem:Inter leq FD}
Let $(\X,f),(\Y,g) \in \Reeb$.
Then 
\begin{equation*}
 \inter(f,g) \leq d_{FD}(f,g).
\end{equation*}

\end{lemma}

\begin{proof}
Let $\e > d_{FD}(f,g)$.
By definition of the functional distortion metric, there are maps
 \begin{equation*}
  \begin{tikzcd}
   \X \ar[start anchor=north east, end anchor=north west]{r}{\phi}
          & \Y \ar[start anchor=south west, end anchor=south east]{l}{\psi}
  \end{tikzcd}
 \end{equation*}
which satisfy the requirements of Definition~\ref{Def:FuncDistortion}, in particular, $\max\{\|f-g\phi\|_\infty, \|g-f\psi\|_\infty\} \leq \e$.
Then it is obvious that $\phi(f\inv(I)) \subset g\inv(I^\e)$ and $\psi(g\inv(I)) \subset f\inv(I_\e)$ for all intervals $I \subset \R$.
This implies that there is a map induced by $\phi$
\begin{equation*}
 \begin{array}{rccc}
  \phi^*_I: &  \pi_0(f\inv(I)) & \longrightarrow & \pi_0(g\inv(I^\e))
 \end{array}
\end{equation*}
and a map induced by $\psi$
\begin{equation*}
 \begin{array}{rccc}
  \psi^*_I: &  \pi_0(g\inv(I)) & \longrightarrow & \pi_0(f\inv(I^\e))
 \end{array}
\end{equation*}
for every $I$.
By the functoriality of $\pi_0$, these maps commute with the maps induced by the inclusions $\iota: I \hookrightarrow J$ and $\iota_\e: I^\e \hookrightarrow J^\e$; that is, the diagrams
\begin{equation*}
 \begin{tikzcd}
 \pi_0(f\inv(I)) \ar{r}{\phi^*_I} \ar{d}{\iota^*}&  \pi_0(g\inv(I^\e))  \ar{d}{\iota_\e^*} &&
 \pi_0(g\inv(I)) \ar{r}{\psi^*_I} \ar{d}{\iota^*}&  \pi_0(f\inv(I^\e))  \ar{d}{\iota_\e^*}\\
 \pi_0(f\inv(J)) \ar{r}{\phi^*_J}&  \pi_0(g\inv(J^\e)) &&
 \pi_0(g\inv(J)) \ar{r}{\psi^*_J}&  \pi_0(f\inv(J^\e)) \\
 \end{tikzcd}
\end{equation*}
commute.
In addition, the functoriality of $\pi_0$ implies that the diagrams
\begin{equation*}
 \begin{tikzcd}
 \pi_0f\inv(I) \ar{r}\ar{dr} & \pi_0 g\inv(I^\e) \ar{d} &&  \pi_0g\inv(I) \ar{r} \ar{dr} & \pi_0 f\inv(I^\e) \ar{d} \\
 & \pi_0f\inv(I^{2\e}) && & \pi_0g\inv(I^{2\e}) 
\end{tikzcd}
\end{equation*}
commute.
These are exactly the properties necessary to call $\phi^*$ and $\psi^*$ an $\e$-interleaving of the associated cosheaves. Since the above holds for any $\e > d_{FD}(f,g)$, we conclude $\inter(f,g) \leq \e = d_{FD}(f,g)$.
\end{proof}

\subsection{The Hard Direction}

In order to show $d_{FD}((\X,f), (\Y,g)) \leq \inter((\X,f), (\Y,g)) $, we need to start with an $\e$-interleaving,  $\phi:(\X,f) \to \TT_\e(\Y,g)$ and $\psi:(\Y,g) \to \TT_\e(\X,f)$, and construct a pair of maps satisfying the requirements of the functional distortion distance.
To do this, note that the map $\phi$ induces a multimap $\bar \phi: \X \to \Y_\e$, which sends a point $x$ to the entire equivalence class of $\phi(x)$, thought of as a subset of $\Y_\e$.
Concretely, letting $q:\Y_\e \to \TT_\e(\Y)$ denote the Reeb graph quotient map, we have $\bar\phi = q^{-1}\circ \phi$.
For any $\delta$, we can construct the multimap, $\bar\phi_\delta:\X \to \Y_\e$ which sends $x$ to $\bar\phi(B_\delta(x))$, where $B_\delta(x) = \{x' \mid d_f(x,x')<\delta\}$. Explicitly, we have
\begin{equation*}
\bar\phi_\delta(x) = \{(y',t') \in \Y_\e \mid (y',t') \in \bar\phi(x'), d_f(x,x')<\delta\},
\end{equation*}
See Fig.~\ref{Fig:DeltaMap} for an example.

\begin{figure}
 \centering
 \includegraphics[scale = .8]{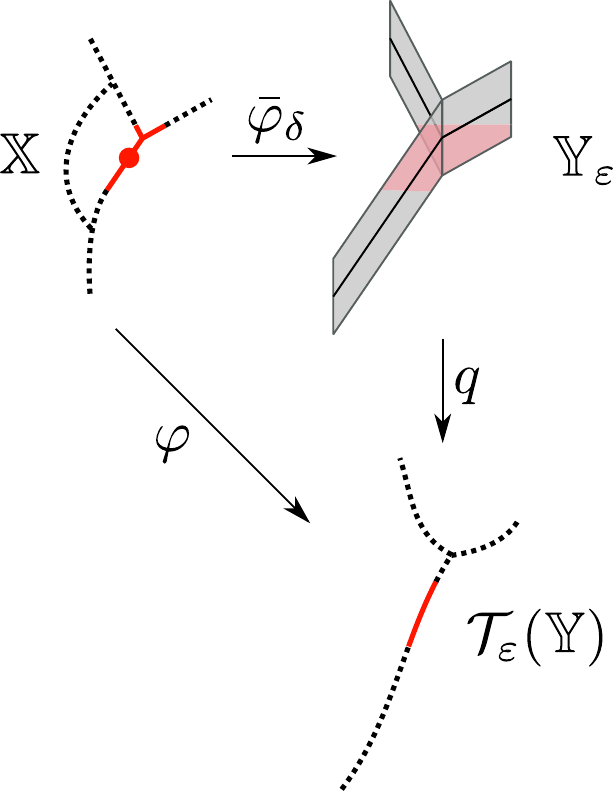}
 \caption{An example for determining the map $\bar\phi_\delta$.  Given the red point $x \in \X$, the red solid region in $\X$ is $B_\delta(x)$.  Then we can look at $\phi(B_\delta(x))$, the red region in $\TT_\e(\Y)$.  The set $\bar\phi_\delta(x)$ in $\Y_\e$ consists of the points which map into $\phi(B_\delta(x))$ in $\TT_\e(\Y)$ under the quotient map $q$. }
 \label{Fig:DeltaMap}
\end{figure}

We want to show that the map $\bar \phi_\delta: \X \to \Y_\e$ satisfies the assumptions of Theorem~\ref{Thm:Selection}. 
\begin{enumerate}
 \item 
 Since $\X$ is a finite graph, it is compact and thus trivially paracompact.
In addition, because it is a graph, it has covering dimension 1. 
\item
We can embed the graph $\Y$ in $\R^3$.  
In this context, it is a closed subset of a completely metrizable space, and thus is completely metrizable.
Therefore, $\Y_\e$ is also completely metrizable since it is the product of two completely metrizable spaces.
\item  To show that $\bar\phi_\delta$ is LSC, let $U \subset \Y_\e$ open and $x \in \bar\phi_\delta\inv(U)$ given.
 This means that there is an $x' \in B_\delta(x)$ such that $\bar\phi(x')\cap U \neq \emptyset$.
 Let $r>0$ be a radius such that $B_r(x')\subset B_\delta(x)$.  
 In particular, this means that $d_f(x,x')<\delta-r$.
 We now want to show that $B_r(x) \subset \bar\phi_\delta\inv(U)$.
 Let $x'' \in B_r(x)$.  
 We know that $x' \in B_\delta(x'')$ since 
 \begin{equation*}
  d_f(x',x'') \leq d_f(x',x) + d_f(x,x'') <r + (\delta-r) = \delta.
 \end{equation*}
 Then since $\bar\phi(x')\cap U \neq \emptyset$ and $x' \in B_\delta(x'')$, we must have $x'' \in \bar\phi_\delta\inv(U)$.

\item 
Let $q:\Y_\e \to \TT_\e(\Y)$ be the quotient map.  
Then $q\circ\pbd(x) = \phi(B_\delta(x))$ is the image of a connected component under a continuous map and is therefore connected.
Since $\phi(x) \subset \TT_\e(\Y)$ is by definition the image of a connected component of $\X$, it is also connected.
So $\pbd(x)$ can be thought of as a fibration with base space $\phi(B_\delta(x))$ and fibers $\pb(x')$.
Since the fibers are connected by definition, and the base is connected, the total space is connected.

\item
As checking this property is by far the most complicated, we prove it in Lemma \ref{Lem:ELC}.

\end{enumerate}

\begin{lemma}
\label{Lem:ELC}
 The family of values $\{\pbd(x)\}_{ x \in \X}$ is $\ELC^0$.
\end{lemma}
\begin{proof}
Fix $x \in \X$, and consider $\pbd(x) \subset \Y_\e$.  
Choose $(y,t) \in \pbd(x)$ and let $W$ be a neighborhood of $(y,t)$.  
Let $\bar r$ such that $B_{\bar r}(y,t)$, the ball of radius $\bar r$ using the metric 
\begin{equation*}
d_{\Y_\e}((y,t),(y',t')) = \max(d_g(y,y'),|t-t'|), 
\end{equation*}
is contained in $W$.
Let $L$ be the minimum height of any edge in $\TT_\e(\Y)$.
Then set $r = \min(\bar r, L/4)$ and let $V = B_r(y,t)$.

We need show that for any $\tx$ such that $\pbd(\tx) \cap V \neq \emptyset$, $\pbd(\tx) \cap V$ is path connected.  
Let $(y_1,t_1)$ and $(y_2,t_2)$ be  in the intersection and, seeking a contradiction, assume that they are in different path components of $U \cap V$.
Then since $U$ is path connected, there is a path $\gamma_1 $ from $(y_1,t_1)$ to $(y_2,t_2)$ which stays completely inside of $U$.
Because this path is in $U = \pbd(\tx)$, there is an $x_s\in B_\delta(\tx)$ such that $\gamma_1(s) \in \pb(x_s)$ and so $g(\gamma_1(s)) = f(x_s)$.  
Since $x_s \in B_\delta(\tx)$, this implies that $g(\gamma_1(s)) \in B_\delta( f(\tx))$ and thus $\height(\gamma_1)<2\delta$. 
As $V$ is path connected, there is a path $\gamma_2$ from $(y_2,t_2)$ to $(y_1,t_1)$ which stays completely inside of $V$.
So, $g(\gamma_2(s)) \in B_r(y,t)$ and thus  $\height(\gamma_2)<2r$.

We can now consider the paths $q(\gamma_1)$ and $q(\gamma_2)$ in $\TT_\e(\Y)$.
As $U \cap V$ is not path connected, 
 there is a point $v \in\gamma_2$ which is not in $U$.
We want to show that $q(v) \not \in q(\gamma_1)$.
By definition, $\bar\phi$ is the map such that $q\bar\phi(z) = \phi(z)$ for any $z \in \X$.
Thus
\begin{equation*}
 q(U) = q(\pbd(\tx)) = q(\bar\phi(B_\delta(x))) = \phi(B_\delta(x)).
\end{equation*}
Again seeking a contradiction, if $q(v) \in q(U) = \phi(B_\delta(\tx))$, there is an $x_v \in B_\delta(\tx)$ such that $\phi(x_v) = q(v)$.
But this implies that $v \in \bar\phi(x_v)$ and thus $v \in \pbd(\tx) = U$, contradicting our original assumption, so we must have $q(v) \not \in q(\gamma_1) \subset q(U)$.

This implies that the loop $q(\gamma_1\gamma_2)$ is nontrivial in $\TT_\e(\Y)$.
However, $$\height(p(\gamma)) < 2r + 2\delta < L,$$ and therefore cannot go around a loop in $\TT_\e(\Y)$.
Thus, the loop cannot be nontrivial, so the original assumption that $\pbd(\tx)\cap V$ is not path connected must be false.
\end{proof}

Thus, since $\bar \phi$ satisfies the requirements for Theorem~\ref{Thm:Selection}, there exists a continuous selection $\widetilde \phi_\delta:\X \to \Y_\e$ of $\pbd$, i.e., a map $\tilde \phi_\delta:\X \to \Y_\e$ satisfying $\tilde \phi_\delta(x) \in \bar\phi(B_\delta(x))$ for all $x \in \X$.
Likewise, there exists a continuous selection $\widetilde\psi_\delta:\Y \to \X_\e$ for $\bar \psi_\delta$.
Let $p_1$ be either the map $\X_\e \to \X$ or $\Y_\e \to \Y$, defined by projection onto the first factor.
Then we will define our maps for the functional distortion distance to be $\Phi = p_1 \circ \widetilde \phi_\delta:\X \to \Y$ and $\Psi = p_1 \circ\widetilde \psi_\delta:\Y \to \X$.
Note that $\Phi$ and $\Psi$ depend on the choice of $\delta$.

In order to prove the main result of this section, Lemma \ref{Lem:FD leq Inter}, we need the following two technical lemmas.
Let $\iota: (\X,f) \to \TT_\e(\X,f)$ be the map which maps $x$ to the class of $(x,0)$. Moreover, let $\kappa=p_1q^{-1}$. Note that $\iota^{-1} \subseteq \kappa$. Similarly, define $\iota_\e: \TT_\e(\X,f) \to \TT_{2\e}(\X,f)$ and $\kappa_\e=p_1q_\e^{-1}$, where $q_\e$ is the quotient map $\TT_\e(\X,f)\times[-\e,\e] \to \TT_\e(\X,f)$.

\begin{lemma}
\label{Lem:IotaBalls}
 $\kappa B_r \iota(x) \subseteq B_{r+2\e}(x)$.
\end{lemma}
\begin{proof}
 Let $x \in \X$ be given, and let $(\tx,\tilde t) \in B_r(\iota(x))$.
 We want to show that $\tx \in B_{r+2\e}(x)$.
 
 First, $\iota(x)=[x,0]$ is by definition the class of $(x,0)$ in $\X_\e$.
 Since $(\tx,\tilde t) \in B_r(\iota(x))$, there is a path $\gamma$ of height at most $r$ between $[x,0]$ and $[\tx,\tilde t]$ in $\TT_\e(\X)$.
 Consider the set  $q\inv(\gamma) \subset \X_\e$.  
 As the $g_\e$ function values of the points in $q\inv(\gamma)$ are the same as the $g_\e$ function values of the corresponding points in $\gamma$, we know $g_\e(q\inv(\gamma)) = g_\e(\gamma)$.
 Then since $g(x')$ is within $\e$ of $g_\e(x',t)$ for each $x'$, we have $g(\kappa(\gamma)) = g(p_1q\inv(\gamma)) \subset B_\e(g_\e(\gamma))$.
 
 Because the image of $q\inv(\gamma)$ under $q$ is path connected and $q$ preserves path components, $q\inv(\gamma)$ is path connected as well.
 In particular, $(x,0)$ and $(\tx,\tilde t)$ are in this set, so there is a path $\zeta$ between them.
 
 Breaking up $\zeta$ into its $\X$ and its $[\e,\e]$ component, consider $\zeta(s) = (\zeta_\X(s),\zeta_\e(s))$.
 Then $\zeta_\X(s)$ is a path in $\X$ between $x$ and $\tx$.  
 As it is contained in $q\inv(\gamma)$, this also means that $\height(\zeta_\X) \leq \height(\gamma)+2\e \leq r+2\e$, and so $\tx \in B_{r+2\e}(x)$.
\end{proof}
Note that the previous lemma can also be stated using $\iota_\e$, so $\kappa_\e B_r \iota_\e \subseteq B_{r+2\e}$.
Since this lemma works for $r=0$, this also implies that $\kappa\iota \subset B_{2\e}$.

\begin{lemma}
\label{Lem:PhiPsiBall}
  $\Psi\Phi(x) \in B_{6\e+2\delta} (x)$.
\end{lemma}

\begin{proof}
By definition of $\Phi$ and $\Psi$, we have
\begin{align*}
\Phi(x) & \in \kappa \phi B_\delta(x),\\
\Psi(y) & \in \kappa \psi B_\delta(y),
\end{align*}
so
\begin{equation*}
\Psi  \Phi(x) \in \kappa  \psi  B_\delta \kappa \phi B_\delta (x).
\end{equation*}
Since $\psi$ preserves function values, a path $\gamma$ in $\Y$ is sent to a path $\psi(\gamma)$ of the same height.
Thus $\psi B_\delta \subseteq B_\delta \psi$ and we get
\begin{equation*}
\Psi \Phi(x) \in \kappa B_\delta  \psi \kappa B_\delta  \phi  (x).
\end{equation*}
Now using, Lemma \ref{Lem:IotaBalls}, the definition of the interleaving, and the fact that for any map $\nu$, $x \in \nu\inv\nu(x)$, we have
\begin{align*}
 \psi  \kappa  \phi 
& \subseteq \iota\inv_\e  \iota_\e  \psi  \kappa \phi \\
& \subseteq \kappa_\e  \iota_\e  \psi  \kappa \phi \\
& = \kappa_\e  \psi_\e  \iota \kappa  \phi \\
& \subseteq \kappa_\e \psi_\e B_{2\e} \phi \\
& \subseteq \kappa_\e B_\e \psi_\e \phi \\
& = \kappa_\e  B_\e \iota_\e \iota \\
& \subseteq B_{4\e} \iota,
\end{align*}
so we have 
\begin{align*}
 \Psi  \Phi(x) 
 &\in \kappa B_\delta \psi  \kappa \phi  B_\delta (x) \\
 &\subseteq \kappa B_\delta B_{4\e} \iota  B_\delta (x) \\
 &\subseteq \kappa B_\delta  B_{4\e}  B_\delta \iota (x) \\
 &= \kappa B_{4\e+2\delta} \iota (x) \\
 &\subseteq B_{6\e+2\delta} (x). \qedhere
\end{align*}
\end{proof}

Finally, we can  prove the main result of the section.

\begin{lemma}
\label{Lem:FD leq Inter}
Let $f:\X \to \R$ and $g:\Y \to \R$.  
Then 
\begin{equation*}
 d_{FD}(f,g) \leq 5\inter(f,g).
\end{equation*}

\end{lemma}

\begin{proof}
Let $\phi:(\X,f) \to \TT_\e(\Y,g)$ and $\psi:(\Y,g) \to \TT_\e(\X,f)$ be an $\e$-interleaving, and thus $\inter(f,g) \leq \e$. 
As proved above, there exists continuous selections $\Phi: \X \to \Y$ and $\Psi: \Y \to \X$ for the multimaps $\bar \phi_\delta$ and $\bar \psi_\delta$.
In particular, this means that for any $x \in \X$, $\Phi(x)$ is a point in $\Y$ such that there is a $x' \in \X$ and a $t \in [-\e,\e]$ with $(\Phi(x),t) \in \bar\phi(x')$ and $d_f(x,x')<\delta$.
So $|f(x)-f(x')|<\delta$ and $f(x') = g(\Phi(x))+t$.
Thus 
\begin{equation*}
 |f(x)-g\Phi(x)| = |f(x)-(f(x')-t)| = |f(x)-f(x')+t| \leq \delta + \e
\end{equation*}
 and hence $\|f-g \circ\Phi\|_\infty \leq \e + \delta$.
Likewise, $\|g-f \circ \Psi\|_\infty \leq \e + \delta$.

 Consider $(x,y),(x',y') \in C(\Phi,\Psi)$.  
 There are two cases to consider; either the pairs are the same type (i.e. $(x,\Phi(x))$ and $(x', \Phi(x'))$), or they are different.
 If we have the same type, so $(x,\Phi(x))$ and $(x', \Phi(x'))$, let $\gamma$ be a minimum height path in $\X$ from $x$ to $x'$.  
 Then $\Phi(\gamma)$ is a path in $\Y$ from $\Phi(x)$ to $\Phi(x')$.
 Since $\|f - g\Phi\|_\infty\leq \e + \delta$, the height of $\Phi(\gamma)$ is at most $2(\e+\delta)$ above the height of $\gamma$.  
 So 
 \begin{equation}
 \label{Eqn:BoundA1}
 \begin{aligned}
  d_g(\Phi(x),\Phi(x')) 
    &\leq \height(\Phi(\gamma)) \\
    &\leq \height(\gamma) + 2(\e + \delta)\\
    &= d_f(x,x') + 2(\e +\delta).
 \end{aligned}
 \end{equation}
 
 To get the other direction, let $\zeta$ be a minimum height path in $\Y$ between $\Phi(x)$ and $\Phi(x')$.  
 Then $\Psi(\zeta)$ is a path in $\X$ from $\Psi\Phi(x)$ to  $\Psi\Phi(x')$.
 A similar argument as above gives
 \begin{equation*}
 \begin{aligned}
  d_f (\Psi\Phi(x),\Psi\Phi(x')) 
    &\leq \height(\Psi(\zeta))\\
    &\leq \height(\zeta) + 2(\e+\delta) \\
	&= d_g(\Phi(x),\Phi(x')) + 2(\e+\delta).
 \end{aligned}
 \end{equation*}
 Thus, using Lemma \ref{Lem:PhiPsiBall} and the triangle inequality,
 \begin{equation}
 \label{Eqn:BoundA2}
\begin{aligned}
 d_f(x,x')& \leq d_f(x,\Psi\Phi(x)) +  d_f(\Psi\Phi(x),\Psi\Phi(x')) +  d_f(\Psi\Phi(x'),x')  \\
      &\leq  d_g(\Phi(x),\Phi(x')) + 2(\e+\delta) + 2(6\e+2\delta)\\
      &\leq  d_g(\Phi(x),\Phi(x')) + 14(\e+\delta)
\end{aligned}
\end{equation}
and therefore combining Eqns.~\ref{Eqn:BoundA1} and \ref{Eqn:BoundA2},
\begin{equation*}
 |d_f(x,x')-d_g(\Phi(x),\Phi(x'))|\leq 14(\e+\delta).
\end{equation*}
Likewise, 
\begin{equation*}
  |d_f(\Psi(y),\Psi(y'))-d_g(y,y')|\leq  14(\e+\delta).
\end{equation*}

 Now assume we have $(x,\Phi(x))$ and $(\Psi(y), y)$ and let $\zeta$ be a minimum height path in $\Y$ between $\Phi(x)$ and $y$.
 Then $\Psi(\zeta)$ is a path in $\X$ between $\Psi\Phi(x)$ and $\Psi(y)$.
 Thus,
 \begin{equation*}
 \begin{aligned}
  d_f(\Psi(y),\Psi\Phi(x)) 
      &\leq \height(\Psi(\zeta)) \\
      & \leq \height(\zeta)+2(\e+\delta) \\
      &= d_g(\Phi(x),y) + 2(\e+\delta)
 \end{aligned}
 \end{equation*}
 and  using  Lemma \ref{Lem:PhiPsiBall} we have
\begin{equation}
\label{Eqn:BoundB1}
 \begin{aligned}
 d_f(x,\Psi(y))
 & \leq d_f(x,\Psi\Phi(x)) +  d_f(\Psi\Phi(x),\Psi(y))  \\
 &\leq  d_g(\Phi(x),y) + 8(\e+\delta).
 \end{aligned}
\end{equation}
Likewise, if we have a minimum height path $\gamma$ in $\X$ between $x$ and $\Psi(y)$, then $\Phi(\gamma)$ is a path in $\Y$ between $\Phi(x)$ and $\Phi\Psi(y)$ with height bounded by $\height(\gamma)+2(\e+\delta)$, and thus
\begin{equation}
\label{Eqn:BoundB2}
 \begin{aligned}
 d_g(\Phi(x),y) 
 &\leq d_g(\Phi(x),\Phi\Psi(y)) +  d_g(\Phi\Psi(y),y)  \\
 &\leq  d_f(x,\Psi(y)) + 8(\e+\delta).
 \end{aligned}
\end{equation}
Therefore, $| d_f(x,\Psi(y)) - d_g(\Phi(x),y) | \leq 8(\e+\delta)$.

 Combining all of these bounds gives 
 \begin{equation*}
  D(\phi,\psi)  = \sup_{\substack{(x,y),(x',y')\\ \in C(\phi,\psi)}} \frac{1}{2} \left| d_f(x,x') - d_g(y,y')\right| \leq 7(\e+\delta).
 \end{equation*}
 and therefore
 \begin{equation*}
  d_{FD}(f,g) = \inf_{\phi,\psi} \max \{ D(\phi,\psi),\|f-g \circ \phi\|_\infty, \|g-f \circ \psi\|_\infty \} \leq 7(\e+\delta).
 \end{equation*} 
 Since this is true for any $\e>\inter(f,g)$ and for any $\delta>0$, this completes the proof.
\end{proof}

Putting together Lemmas \ref{Lem:Inter leq FD} and \ref{Lem:FD leq Inter}, our main theorem is immediate.
\begin{theorem}
\label{Thm:Main}
 The functional distortion metric and the interleaving metric are strongly equivalent.  That is, given any Reeb graphs $(\X,f)$ and $(\Y,g)$,
 \begin{equation*}
  \inter(f,g) \leq d_{FD}(f,g) \leq 7\inter(f,g).
 \end{equation*}
\end{theorem}

\section{Relationship Between the Interleaving and Bottleneck Distances}
\label{sec:bottleneck}

Having strongly equivalent metrics means that we can quickly pass back and forth many of the properties associated to the metrics.
For example, the bottleneck stability bound for persistence diagrams in terms of the functional distortion distance~\cite{Bauer2014} says the following (for the definitions of the persistence diagrams $\Dg_0(f), \ExDg_1(f)$ associated to a function~$f$ and of the bottleneck distance $d_B$ we refer the reader to~\cite{EH09}):
\begin{theorem}
Given two Reeb graphs $(\X,f)$ and $(\Y,g)$,
\begin{equation*}
 d_B(\Dg_0(f),\Dg_0(g)) \leq d_{FD}(f,g)
\end{equation*}
 and 
\begin{equation*}
 d_B(\ExDg_1(f),\ExDg_1(g)) \leq 3d_{FD}(f,g).
\end{equation*}
\end{theorem} 
 
Combining this result with Theorem~\ref{Thm:Main} gives an immediate stability result relating the interleaving distance with the bottleneck distance.
\begin{theorem}
Given two Reeb graphs $(\X,f)$ and $(\Y,g)$,
\begin{equation*}
 d_B(\Dg_0(f),\Dg_0(g)) \leq 7\inter(f,g)
\end{equation*}
 and 
\begin{equation*}
 d_B(\ExDg_1(f),\ExDg_1(g)) \leq 21\inter(f,g).
\end{equation*}
 
\end{theorem}

\section{Discussion}

In this paper, we study the relation between the two existing distance measures for Reeb graphs, and show that they are strongly equivalent on the space of Reeb graphs. 
This relationship will be a powerful tool for understanding convergence properties of the different metrics. 
For example, if we have a Cauchy sequence in one metric, we have a Cauchy sequence in the other and can therefore pass around completeness results. 
This relationship also means that algorithms for computation and approximation of the metrics can be written using whichever method is most helpful and applicable to the context.

These two distances in general may not be the same. An immediate question is whether the relations provided in Theorem \ref{Thm:Main} are tight. 
In particular, it is easy to construct examples where the bound $\inter(f,g) \leq d_{FD}(f,g)$ of Lemma~\ref{Lem:Inter leq FD} is tight; it will be interesting to investigate the tightness of the bound $d_{FD}(f,g) \leq 7\inter(f,g)$ of Lemma~\ref{Lem:FD leq Inter}. 
While that bound is obtained using an arbitrary selection, a better bound may be achievable using a particular optimal selection. 
In addition, this may shed light on whether the bounds given between the bottleneck distance 
of the extended persistence diagrams and the two Reeb graph distances are tight. 
Finally, we will explore the applications of these distance measures to studying the stability of Reeb-like structures, such as Mapper and $\alpha$-Reeb graphs.

\appendix

\section{Appendix: The Category of Constructible Cosheaves}
\label{Sect:Cosheaves}
It has been shown in \cite{deSilva2014} that the category $\Reeb$ is equivalent to a particular class of cosheaves.
This allows a definition of distance for cosheaves to be pulled back to a definition of distance for Reeb graphs.  
A brief overview of the necessary sheaf theory follows; a better introduction can be found in \cite{Curry2013}.

 A \textit{category} $\CatA$ is a collection of objects $A \in \CatA$ with morphisms $f:A \to A'$ between the objects.
 We also require a composition operation which is associative, as well as an identity morphism $\mathbb{1}_A: A \to A$ for each $A \in \CatA$.
 There are several categories used in the definition of the interleaving distance.  
 They are notated as follows:
 \begin{itemize}
  \item $\Set$ consists of sets with morphisms given by set maps.
  \item $\Top$ consists of topological spaces with continuous maps.
  \item $\Int$ consists of open intervals $I \subset \R$ with a unique morphism $I \to J$ iff $I \subseteq J$.
 \end{itemize}
 
A \textit{functor} $F: \CatA \to \CatB$ is a map between categories that sends each object $A$ to an object $F(A)$ and each morphism $f: A \to A'$ to a morphism $F[f]:F(A) \to F(A')$.
We require that the functor respects composition and identities, so $F(\mathbb{1}_A) = \mathbb{1}_{F(A)}$ for every $A \in \CatA$, and $F[g\circ f] = F[g] \circ F[f]$ for every pair of morphisms $f:A \to A'$ and $g: A' \to A''$.

Finally, we have the notion of a \textit{natural transformation} $\eta : F \to G$ between functors $F,G:\CatA \to \CatB$.
It is a collection of morphisms $\eta_A: F(A) \to G(A)$, one for each $A \in \CatA$, such that for any morphism $f:A \to A'$ in $\CatA$,
\begin{equation*}
\begin{tikzcd}
 	\Ffunc (A) \ar{r}{\Ffunc {[}f{]}} \ar{d}{\eta_A}]
	&
	\Ffunc (A') \ar{d}{\eta_{A'}}]
	\\
 	\Gfunc (A) \ar{r}{\Gfunc {[}f{]}}]
	&
	\Gfunc (A').
\end{tikzcd}
\end{equation*}
commutes.

Then a set-valued pre-cosheaf, in our context, is a functor $F: \Int \to \Set$. 
A cosheaf is a pre-cosheaf which satisfies the following property.
For any $\mathcal{I}$, a collection of open intervals whose union is an open interval $U$, 
$\Ffunc(U)$ must be the colimit of the diagram 
\begin{equation*}
 \coprod_{I,J \in \mathcal{I}} \Ffunc(I \cap J) \rightrightarrows \coprod_{I \in \mathcal{I}} \Ffunc(I).
\end{equation*}
A cosheaf $\Ffunc:\Int \to \Set$ is constructible if each $\Ffunc(I)$ is finite and there is a finite set of critical values $S \subset \R$ such that
\begin{itemize}
 \item if $I \subseteq J$ are open intervals with $I \cap S = J \cap S$, then $\Ffunc[I \subseteq J]$ is an isomorphism, and
 \item if $I$ is contained in $(-\infty, \min(S))$ or $(\max(S),\infty)$, $\Ffunc(I) = \emptyset$.
\end{itemize}
The category of constructible cosheaves with morphisms given by natural transformations is denoted $\Cshc$.

Will of these definitions in hand, we have the following theorem from \cite{deSilva2014}.
\begin{theorem}
 The categories $\Reeb$ and $\Cshc$ are equivalent.
\end{theorem}
The intuition for the equivalence comes from the following observation.  
Given a Reeb graph $(\tilde \X,\tilde f)$ of a space $(\X,f)$, the level set $f\inv(a)$ has a point for each connected component of the original space $\X$.  
This can be thought of as the following composition of functors:
\begin{equation*}
\begin{tikzcd}
 \Int \ar{r}{f\inv} & \Top \ar{r}{\pi_0} & \Set
\end{tikzcd}
\end{equation*}
where $f\inv$ takes an interval $I$ to $f\inv(I)\subset \X$ and $\pi_0$ sends a topological space to the set of its connected components.
Given a Reeb graph $(\X,f)$, we call  $\Ffunc: \Int \to \Set$ defined by $\Ffunc(I) = \pi_0f\inv(I)$ its associated cosheaf.

\bibliography{ReebGraphs}

\begin{thebibliography}{10}

\bibitem{Bauer2014}
Ulrich Bauer, Xiaoyin Ge, and Yusu Wang.
\newblock Measuring distance between {R}eeb graphs.
\newblock In {\em Proceedings of the Thirtieth Annual Symposium on
  Computational Geometry -- SoCG '14, Kyoto, Japan}, 2014.

\bibitem{Biasotti2008}
S.~Biasotti, D.~Giorgi, M.~Spagnuolo, and B.~Falcidieno.
\newblock Reeb graphs for shape analysis and applications.
\newblock {\em Theoretical Computer Science: Computational Algebraic Geometry
  and Applications}, 392(13):5 -- 22, 2008.

\bibitem{BBKSS13}
Kevin Buchin, Maike Buchin, Marc van Kreveld, Bettina Speckmann, and Frank
  Staals.
\newblock Trajectory grouping structure.
\newblock In {\em Proc. Algorithms and Data Structures Symposium (WADS)},
  volume 8037, pages 219--230, 2013.

\bibitem{Chazal2014}
Fr{\'e}d{\'e}ric Chazal and Jian Sun.
\newblock Gromov-{H}ausdorff approximation of filament structure using
  {R}eeb-type graph.
\newblock In {\em Proceedings of the Thirtieth Annual Symposium on
  Computational Geometry}, SOCG'14, pages 491:491--491:500, New York, NY, USA,
  2014. ACM.

\bibitem{Curry2013}
Justin Curry.
\newblock Sheaves, cosheaves and applications.
\newblock {\em arXiv: 5523.3031}, 2013.

\bibitem{deSilva2014}
Vin de~Silva, Elizabeth Munch, and Amit Patel.
\newblock Categorification of {R}eeb graphs.
\newblock In preparation, 2014.

\bibitem{Dey2013}
Tamal~K. Dey, Fengtao Fan, and Yusu Wang.
\newblock An efficient computation of handle and tunnel loops via {R}eeb
  graphs.
\newblock {\em ACM Trans. Graph.}, 32(4):32:1--32:10, July 2013.

\bibitem{DiFabio2014b}
Barbara Di~Fabio and Claudia Landi.
\newblock The edit distance for {R}eeb graphs of surfaces.
\newblock {\em arXiv: 1411.1544}, 2014.

\bibitem{Doraiswamy2012}
Harish Doraiswamy and Vijay Natarajan.
\newblock Output-sensitive construction of {R}eeb graphs.
\newblock {\em IEEE Transactions on Visualization and Computer Graphics},
  18(1):146--159, 2012.

\bibitem{EH09}
H.~Edelsbrunner and J.~Harer.
\newblock {\em Computational Topology: {An} Introduction}.
\newblock Amer. Math. Soc., Providence, Rhode Island, 2009.

\bibitem{Escolano2013}
Francisco Escolano, Edwin~R. Hancock, and Silvia Biasotti.
\newblock Complexity fusion for indexing {R}eeb digraphs.
\newblock In Richard Wilson, Edwin Hancock, Adrian Bors, and William Smith,
  editors, {\em Computer Analysis of Images and Patterns}, volume 8047 of {\em
  Lecture Notes in Computer Science}, pages 120--127. Springer Berlin
  Heidelberg, 2013.

\bibitem{Ge2011}
Xiaoyin Ge, Issam~I. Safa, Mikhail Belkin, and Yusu Wang.
\newblock Data skeletonization via {R}eeb graphs.
\newblock In J.~Shawe-Taylor, R.S. Zemel, P.~Bartlett, F.C.N. Pereira, and K.Q.
  Weinberger, editors, {\em Advances in Neural Information Processing Systems
  24}, pages 837--845. 2011.

\bibitem{Harvey2010}
William Harvey, Yusu Wang, and Rephael Wenger.
\newblock A randomized {$O(m \log m)$} time algorithm for computing {Reeb}
  graphs of arbitrary simplicial complexes.
\newblock In {\em Proceedings of the 2010 annual symposium on Computational
  geometry}, SoCG '10, pages 267--276, New York, NY, USA, 2010. ACM.

\bibitem{HA03}
Franck H\'{e}troy and Dominique Attali.
\newblock Topological quadrangulations of closed triangulated surfaces using
  the {Reeb} graph.
\newblock {\em Graph. Models}, 65(1-3):131--148, 2003.

\bibitem{Hilaga2001}
Masaki Hilaga, Yoshihisa Shinagawa, Taku Kohmura, and Tosiyasu~L. Kunii.
\newblock Topology matching for fully automatic similarity estimation of {3D}
  shapes.
\newblock In {\em Proceedings of the 28th annual conference on Computer
  graphics and interactive techniques}, SIGGRAPH '01, pages 203--212, New York,
  NY, USA, 2001. ACM.

\bibitem{Michael1956}
Ernest Michael.
\newblock Continuous selections {II}.
\newblock {\em Annals of Mathematics}, 64(3):pp. 562--580, 1956.

\bibitem{Morozov2013}
Dmitriy Morozov, Kenes Beketayev, and Gunther Weber.
\newblock Interleaving distance between merge trees.
\newblock In {\em Proceedings of TopoInVis}, 2013.

\bibitem{NBPF11}
M.~Natali, S.~Biasotti, G.~Patan\`{e}, and B.~Falcidieno.
\newblock Graph-based representations of point clouds.
\newblock {\em Graphical Models}, 73(5):151 -- 164, 2011.

\bibitem{Nicolau2011}
Monica Nicolau, Arnold~J. Levine, and Gunnar Carlsson.
\newblock Topology based data analysis identifies a subgroup of breast cancers
  with a unique mutational profile and excellent survival.
\newblock {\em Proceedings of the National Academy of Sciences},
  108(17):7265--7270, 2011.

\bibitem{Parsa2012}
Salman Parsa.
\newblock A deterministic {$O(m \log m)$} time algorithm for the {Reeb} graph.
\newblock In {\em Proceedings of the 28th annual ACM symposium on Computational
  geometry}, SoCG '12. ACM, 2012.

\bibitem{Reeb1946}
Georges Reeb.
\newblock Sur les points singuliers d'une forme de pfaff compl\`{e}ment
  int\'{e}grable ou d'une fonction num\'{e}rique.
\newblock {\em Comptes Rendus de L'Acad\'{e}mie ses S\'{e}ances}, 222:847--849,
  1946.

\bibitem{Repovs1998}
Du\u{s}an Repov\u{s} and Pavel~V. Semenov.
\newblock {\em Continuous Selections of Multivalued Mappings}.
\newblock Kluwer Academic Publishers, 1998.

\bibitem{Shinagawa1991}
Yoshihisa Shinagawa, Tosiyasu~L. Kunii, and Yannick~L. Kergosien.
\newblock Surface coding based on {Morse} theory.
\newblock {\em IEEE Comput. Graph. Appl.}, 11(5):66--78, September 1991.

\bibitem{Singh2007}
Gurjeet Singh, Facundo M\'emoli, and Gunnar Carlsson.
\newblock Topological methods for the analysis of high dimensional data sets
  and 3d object recognition.
\newblock In {\em Eurographics Symposium on Point-Based Graphics}, 2007.

\bibitem{Wood2004}
Zo\"{e} Wood, Hugues Hoppe, Mathieu Desbrun, and Peter Schr\"{o}der.
\newblock Removing excess topology from isosurfaces.
\newblock {\em ACM Trans. Graph.}, 23(2):190--208, April 2004.

\bibitem{Yao2009}
Yuan Yao, Jian Sun, Xuhui Huang, Gregory~R. Bowman, Gurjeet Singh, Michael
  Lesnick, Leonidas~J. Guibas, Vijay~S. Pande, and Gunnar Carlsson.
\newblock Topological methods for exploring low-density states in biomolecular
  folding pathways.
\newblock {\em The Journal of Chemical Physics}, 130(14):--, 2009.

\end{thebibliography}

\end{document}